\documentclass[leqno,platex]{article}
  \usepackage{amsmath,amsthm,amssymb,tikz-cd,mathtools,caption,cleveref,pgfplots,tikz,here,enumitem}
  \usetikzlibrary{positioning}
  \usetikzlibrary{arrows}
  \theoremstyle{definition}
  \newtheorem{thm}{Theorem}[section]
  
  \crefname{thm}{Theorem}{Theorems}

  \crefname{prop}{Proposition}{Propositions}
  \newtheorem{lem}[thm]{Lemma}
  
  \crefname{lem}{Lemma}{Lemmas}

  \crefname{fact}{Fact}{Facts}

  \crefname{figure}{figure}{figures}
  \crefformat{enumi}{#2#1#3}

  \newcommand{\lie}[1]{\mathfrak{#1}}
  \newcommand{\bb}[1]{\mathbb{#1}}
  \newcommand{\bbZ}{\bb{Z}}
  \newcommand{\bbR}{\bb{R}}
  \newcommand{\bbC}{\bb{C}}
  \newcommand{\mcl}[1]{\mathcal{#1}}

  \newcommand{\lig}{\lie{g}}
  \newcommand{\lih}{\lie{h}}
  \newcommand{\lisl}{\lie{sl}}
  \newcommand{\lin}{\lie{n}}
  \newcommand{\lis}{\lie{s}}
  \newcommand{\hsp}{\hspace{2em}}
  \newcommand{\lra}[1]{\langle #1 \rangle}
  
  \newcommand{\ovl}[1]{\overline{#1}}
  \newcommand{\klis}{k_\lis}
  \newcommand{\llis}{l_\lis}
  \newcommand{\mlis}{m_\lis}
  \newcommand{\nlis}{n_\lis}
  \newcommand{\plis}{p_\lis}
  \allowdisplaybreaks[4]

  \newlist{newEnum}{enumerate}{3}
  \setlist[newEnum,1]{label=(\arabic*),ref=(\arabic*)}
  \setlist[newEnum,2]{label=(\alph*),
                    ref  =\themyenumi{.(\alph*)}}
  \setlist[newEnum,3]{label=(\arabic*),
                    ref  =\themyenumii{.(\arabic*)}}
  \crefname{newEnumi}{}{}
  \crefname{newEnumii}{}{}
  \crefname{newEnumiii}{}{}
  
  \crefname{equation}{}{}

\title{Irreducible module decompositions
of rank 2 symmetric hyperbolic Kac-Moody Lie algebras
by $\lisl_2$ subalgebras which are generalizations of principal $\lisl_2$ subalgebras}
\author{TSURUSAKI Hisanori\thanks{Graduate School of Mathematical Sciences, University of Tokyo
, htsurusaki1929@gmail.com}}
\date{}

\begin{document}
\maketitle
\setcounter{page}{1}
\begin{abstract}
  There exist principal $\lisl_2$ subalgebras for hyperbolic Kac-Moody Lie algebras.
  In the case of rank 2 symmetric hyperbolic Kac-Moody Lie algebras,
  certain $\lisl_2$ subalgebras are constructed in a previous paper.
  These subalgebras are generalizations of principal $\lisl_2$ subalgebras.
  We show that the rank 2 symmetric hyperbolic
  Kac-Moody Lie algebras themselves are irreducibly decomposed
  under the action of this $\lisl_2$ subalgebras.
  Furthermore, we classify irreducible components of the decomposition.
  In particular, we obtain multiplicities of unitary principal series
  and complementary series.
\end{abstract}
\section{Introduction}
A nilpotent orbit in a finite dimensional simple Lie algebra $\lig_0$
is an orbit obtained by acting on the nilpotent element $x$ of $\lig_0$
by inner automorphisms.
In \cite{dyn}, these are classified by weighted Dynkin diagrams.
From the Jacobson-Morozov theorem, for a nilpotent element $x$ of $\lig_0$,
we can construct a $\lisl_2$-triple with $x$ as a nilpositive element
(\cite[Theorem 3.3.1]{cm}).

This makes it equivalent to classify nilpotent orbits of $\lig_0$ and
to classify $\lisl_2$ triples in $\lig_0$ up to inner automorphisms.
Among the nilpotent orbits of a finite dimensional simple Lie algebra,
the one whose dimension as an algebraic variety is maximal is called
the principal nilpotent orbit.
Correspondingly, we can construct
a principal $SO(3)$ subalgebra that is compatible with compact involution
(\cite{kos}).

Kac-Moody Lie algebras are generalizations of finite-dimensional simple Lie algebras.
They are classified into three types: finite type, affine type,
and indefinite type.
The finite type Kac-Moody Lie algebras are finite dimensional
simple Lie algebras.
Within indefinite Kac-Moody Lie algebra,
there is a class called hyperbolic Kac-Moody Lie algebra.
A hyperbolic Kac-Moody Lie algebra is an indefinite type Kac-Moody
Lie algebra such that any true subdiagram of its Dynkin diagram
is of finite or affine type.
Hyperbolic Kac-Moody Lie algebras, in particular $E_{10}$, are noted to
be related to string theory(\cite{vis}).

By analogy with the above theory, in \cite{no},
for a hyperbolic Kac-Moody Lie algebra, its
principal $SO(1, 2)$ subalgebra was constructed.
Note that \cite{gow} shows that it is possible to
construct a principal $SO(1, 2)$ subalgebra for
certain indefinite Kac-Moody Lie algebra that is not hyperbolic.

Corresponding to this principal $SO(1, 2)$ subalgebra,
we can construct a principal $\lisl_2$-subalgebra
in a hyperbolic Kac-Moody Lie algebra.
In \cite{tsu}, for the rank 2 symmetric hyperbolic Kac-Moody Lie algebras $\lig$,
the following result is obtained.
Let $R_\lig$ be the space that the positive real root vectors span.
We consider $\lisl_2$ subalgebras whose nilpositive element exists
in $R_\lig$.
Then we can construct certain $\lisl_2$ subalgebras.
These subalgebras are generalizations of principal $\lisl_2$ subalgebras.

In this paper, for an $\lisl_2$ subalgebra
of rank 2 symmetric hyperbolic Kac-Moody Lie algebra $\lig$ constructed in \cite{tsu},
we show $\lig$ is decomposed into irreducible $\lisl_2$-modules
by its action on $\lig$.

We are going to more details.
Let $\lis$ be an $\lisl_2$ subalgebra constructed in \cite{tsu}.
Let $H, X, Y$ be an $\lisl_2$ triple and assume that $\lis$ is spanned by $H, X, Y$.
Let $e_i, f_i, h_i ,\; (i = 0, \ldots, n - 1)$ be the Chevalley generators of $\lig$.
Let $\lih_\bbR$ be the $\bbR$-span of $h_i$'s.
From \cite[Theorem 2.2]{kac}, $\lig$ has a $\bbC$-valued nondegenerate invariant symmetric
bilinear form $(\cdot \mid \cdot)$ called the standard form.
An antilinear automorphism $\omega_0$ of $\lig$, called compact involution,
is determined by
\begin{align*}
  \omega_0(e_i) &= -f_i,\\
  \omega_0(f_i) &= - e_i \hsp (i = 0, \ldots, n - 1),\\
  \omega_0(h) &= -h \hsp (h \in \lih_\bbR).
\end{align*}
From \cite[\S 2.7]{kac}, we can determine a nondegenerate Hermitian form
$(\cdot \mid \cdot)_0$ on $\lig$ with $(x \mid y)_0 = - (\omega_0 (x) \mid y)$.

An $\lis$-module $V \subset \lig$ is called unitarizable if following conditions are satisfied.
\begin{newEnum}
  \item $(\cdot \mid \cdot)_0$ on $V$ is positive definite.
  \item For $v_1, v_2 \in V$ and $s \in \lis$, the following condition is satisfied.
  \begin{align*}
    ([s, v_1], v_2)_0 = - (v_1, [\omega_0 (s), v_2])_0.
  \end{align*}
\end{newEnum}
\begin{thm}[\cref{decomposabilityOfG}]
  $\lig$ can be decomposed into a direct sum of irreducible $\lis$-modules
  such that $\lis$ itself is one of the direct summand.
  All of these modules except for $\lis$ are unitarizable.
\end{thm}
Also, we classify how many highest weight modules, lowest weight modules,
and modules that are neither highest weight module nor lowest weight module appear
in this decomposition.
We regard a root $s \alpha_1 + t \alpha_2$ as a point $(s, t)$ in $xy$-plane,
and We define a region $L, -L$ in $xy$-plane in \S 5.
If a root $\alpha$ satisfies $\alpha (H) \in (0, 2)$, $\alpha \in L$.
If a root $\alpha$ satisfies $\alpha (H) \in (-2, 0)$, $\alpha \in -L$.
\begin{thm}[\cref{classification}]
  We consider an irreducible decomposition of $\lig$ by the action of $\lis$.
  \begin{newEnum}
    \item \label{classification1Section1} Let $M$ be an irreducible component of decomposition of $\lig$, which
    contain a root space for a real root in $L$.
    Then, $M$ is a unitary principal or complementary series representation.
    \item \label{classification2Section1} (cf. \cite[Proposition 7.3]{tsu}) There is a unitary principal
    series representation containing an 1-dimensional space in $\lih$.
    \item $\lig$ is decomposed into a direct sum of $\lis$-submodules described in \cref{classification1Section1}
    and \cref{classification2Section1} above, $\lis$ itself, irreducible lowest weight modules,
    and irreducible highest weight modules.
  \end{newEnum}
\end{thm}
We also discuss how to calculate multiplicities of irreducible highest or lowest modules (\S 7).
Furthermore, we classified irreducible components
which are neither highest weight modules nor lowest weight modules,
as either unitary principal or complementary series representations.
\begin{thm}[\cref{AlmostAllsl2ModulesAreComplementary}]
  We consider irreducible components which are neither highest weight modules nor lowest weight modules
  and contain root vectors about real roots in $L$, obtained by \cref{classification}.
  The irreducible components are complementary series representations,
  except those described in \cref{AlmostAllsl2ModulesAreComplementary_IEqualsJ},
  \cref{AlmostAllsl2ModulesAreComplementary_IIsJMinus1}
  and \cref{AlmostAllsl2ModulesAreComplementary_IIsJPlus1}.
  For the exceptions, the irreducible components are unitary principal series representations.
\end{thm}
\section{General theory of Kac-Moody Lie algebras}
Let $\lig$ be a symmetrizable Kac-Moody Lie algebra on $\bbC$.
Let $A$ be the Cartan matrix of $\lig$ and let $A$ be an $n \times n$ matrix.
Let $\lih$ be a Cartan subalgebra of $\lig$.
Let $e_i, f_i, h_i ,\; (i = 0, \ldots, n - 1)$ be the Chevalley generators of $\lig$.
Let $\lih_\bbR$ be the $\bbR$-span of $h_i$'s.

From \cite[Theorem 2.2]{kac}, $\lig$ has a $\bbC$-valued nondegenerate invariant symmetric
bilinear form $(\cdot \mid \cdot)$ called the standard form.

An antilinear automorphism $\omega_0$ of $\lig$, called compact involution,
is determined by
\begin{align*}
  \omega_0(e_i) &= -f_i,\\
  \omega_0(f_i) &= - e_i \hsp (i = 0, \ldots, n - 1),\\
  \omega_0(h) &= -h \hsp (h \in \lih_\bbR).
\end{align*}
From \cite[\S 2.7]{kac}, we can determine a nondegenerate Hermitian form
$(\cdot \mid \cdot)_0$ on $\lig$ with $(x \mid y)_0 = - (\omega_0 (x) \mid y)$.

Write $\lin^+$ for a subalgebra of $\lig$ generated by $e_i$'s and
$\lin^-$ for a subalgebra of $\lig$ generated by $f_i$'s.

We can construct a 3-dimensional subalgebra of $\lig$ which is spanned by three non-zero elements
$J^+ \in \lin^+ ,\; J^- \in \lin^- ,\; J_3 \in \lih$.
$J^+, J^-$ and $J_3$ satisfy
\begin{align*}
  [J_3, J^\pm] &= \pm J^\pm,\\
  [J^+, J^-] &= -J_3.
\end{align*}
This subalgebra is called $SO (1, 2)$ subalgebra of $\lig$.

A representation of $SO(1, 2)$ subalgebra is called unitary if the representation space $V$
has a Hermitian scalar product $(\cdot, \cdot)$ and the following two conditions are satisfied.
\begin{newEnum}
  \item \label{unitary1} The actions of $J^+$ and $J^-$ are adjoint each other, and the action of $J_3$ is
  self-adjoint.
  That is, for $x, y \in V$, we have
    \begin{align*}
      ([J^+, x], y) &= (x, [J^-, y]),\\
      ([J_3, x], y) &= (x, [J_3, y]).
    \end{align*}
  \item Hermitian scalar product $(\cdot, \cdot)$ is positive definite.
\end{newEnum}
When considering the adjoint action of an $SO (1, 2)$ subalgebra of $\lig$ to $\lig$,
from \cite[Lemma 3.1, Lemma 3.2]{tsu},
we can see that the adjoint action satisfying the condition \cref{unitary1} to be unitary
and $J^- = - \omega_0 (J^+)$ are equivalent.
In \cite{no}, principal $SO (1, 2)$ subalgebras for hyperbolic Kac-Moody Lie algebras
are studied. Principal $SO (1, 2)$ subalgebra satisfies that $J^- = - \omega_0 (J^+)$.

When three non-zero elements $X \in \lin^+, Y \in \lin^-, H \in \lih$ of $\lig$
satisfy
\begin{align*}
  [H, X] &= 2X,\\
  [H, Y] &= -2Y,\\
  [X, Y] &= H,
\end{align*}
these three elements are called $\lisl_2$-triple of $\lig$. A $\lig$-subalgebra
that these elements span is called $\lisl_2$ subalgebra.
The $SO(1, 2)$ subalgebras and the $\lisl_2$ subalgebras can be converted by
\begin{align*}
  J^+ &= \frac{1}{\sqrt{2}} X,\\
  J^- &= - \frac{1}{\sqrt{2}} Y,\\
  J_3 &= \frac{1}{2} H.
\end{align*}
The condition $J^- = - \omega_0 (J^+)$ in $SO (1, 2)$ subalgebra is converted to
$Y = \omega_0 (X)$ in $\lisl_2$ subalgebra.
In the following paper, we consider $\lisl_2$ subalgebra
that satisfies $Y = \omega_0 (X)$.
\section{$\lisl_2$-triples of rank 2 hyperbolic symmetric Lie algebra that is compatible to compact involution}
Let $a$ be an integer that satisfies $a \geq 3$, and let $\lig$ be
a hyperbolic Kac-Moody Lie algebra on $\bbC$ such that the Cartan matrix of $\lig$ is
\begin{align*}
  \begin{pmatrix}
    2 & -a\\
    -a & 2
  \end{pmatrix}.
\end{align*}
Let $\alpha_0, \alpha_1$ be the simple roots of $\lig$.
Let $\{ F_n \}$ be the sequence of numbers determined by
$F_0 = 0, F_1 = 1, F_{k + 2} = a F_{k + 1} - F_k$.

\begin{lem}[{\cite[Proposition 4.4]{km}}]
  \label{posroot}

  The real positive roots of $\lig$ are of the form
  \begin{align*}
    \alpha = F_{k + 1} \alpha_0 + F_k \alpha_1
  \end{align*}
  or
  \begin{align*}
    \beta = F_k \alpha_0 + F_{k + 1} \alpha_1.
  \end{align*}
\end{lem}
We distinguish these roots as type $\alpha$ and type $\beta$,
and we also distinguish root vectors belonging to each root as type $\alpha$
and type $\beta$ (cf. \cite[\S 4]{tsu}).

Let $X$ be an element of the space which real positive root vectors span.
Then $X$ can be written as
\begin{align*}
  X = \sum_k c_k E_k, \hsp (k \in \{ 0, \ldots, n_X - 1 \}, c_k \in \bbC, c_k \neq 0, E_k \in \lig_{\beta_k}, E_k \neq 0)
\end{align*}
where $\beta_k \; (k \in \{ 0, \ldots, n_X - 1 \})$ are distinct real roots and
$n_X$ is a positive integer.

We call this $n_X$ the length of $X$. Then the following holds.
\begin{lem}[{\cite[Theorem 5.8]{tsu}}]
  \label{LengthOfXIs2}
  Let $X$ be an element in the space which real positive root vectors span.
  \begin{newEnum}
    \item When the length of $X$ is 1 or more than 3,
    $X ,\; Y = \omega_0(X) ,\; H = [X, Y]$ do not form $\lisl_2$-triple.
    \item \label{length2} Suppose the length of $X$ is 2 and
    $E_0, E_1$ are real positive root vectors of different types
    (in the sense of $\alpha$-type and $\beta$-type).
    Then, taking the appropriate $c_0, c_1 \in \bbC$,
    $X = c_0 E_0 + c_1 E_1 ,\; Y = \omega_0(X)$, and $H = [X, Y]$ form $\lisl_2$-triple.
    In particular, $c_0, c_1$ can be chosen so that $c_0, c_1 \in \bbR$.
  \end{newEnum}
\end{lem}
\begin{lem}[{\cite[Theorem 6.4]{tsu}}]
  \label{dominantLem}
  Take $\lra{H, X, Y}$ in \cref{LengthOfXIs2}, \cref{length2}.
  Let $X = c_0 E_0 + c_1 E_1$, where $E_0$ is type $\alpha$ and
  $E_1$ is type $\beta$.
  From \cref{posroot}, using integers $i, j \geq 0$, we can write
  $E_0 \in \lig_{F_{i + 1} \alpha_0 + F_i \alpha_1} ,\; E_1 \in \lig_{F_{j} \alpha_0 + F_{j + 1} \alpha_1}$.
  If and only if $i = j - 1, j, j + 1$, $H$ is dominant.
\end{lem}
\section{Irreducible decomposition of $\lig$ as an $\lisl_2$ module}
In this section, we consider an $\lisl_2$-subalgebra $\lis = \lra{H, X, Y}$ of $\lig$,
which satisfies the following conditions.
\begin{newEnum}
  \item $H \in \lih$ and $H$ is dominant.
  \item $X$ is in the space which is spanned
  by positive root vectors.
  \item $Y = \omega_0 (X)$.
\end{newEnum}
We show that $\lig$ is decomposed to irreducible modules by the action of $\lis$.

$\lis$-module $V \subset \lig$ is called unitarizable if following conditions are satisfied.
\begin{newEnum}
  \item $(\cdot \mid \cdot)_0$ on $V$ is positive definite.
  \item \label{skewAdjoint} For $v_1, v_2 \in V$ and $s \in \lis$, the following condition is satisfied.
  \begin{align*}
    ([s, v_1], v_2)_0 = - (v_1, [\omega_0 (s), v_2])_0.
  \end{align*}
\end{newEnum}
From \cite[\S 2.7]{kac}, the condition \cref{skewAdjoint} are automatically satisfied.
Therefore, $(\cdot \mid \cdot)_0$ is positive definite on $V$ if and only if $V$ is unitarizable.

First, we put
\begin{align*}
  U = \{ x \in \lig \mid \forall y \in \lis \; (x \mid y)_0 = 0 \}.
\end{align*}
$U$ is closed under the action of $\lis$, and $\lig = \lis \oplus U$.
\begin{lem}
  \label{positiveDefinite}
  $(\cdot \mid \cdot)_0$ is positive definite on $U$.
\end{lem}
\begin{proof}
  From \cite[Theorem 11.7]{kac}, $(\cdot \mid \cdot)_0$
  is positive definite on $\lin^+ \oplus \lin^-$.
  The sign of $(\cdot \mid \cdot)_0$ on $\lih$ is $(1, 1)$.
  Let $\lih_\lis$ be the space that $H$ spans.
  Since $\lis$ itself is not unitarizable,
  when we write $\lih = \lih_\lis \oplus \lih'$,
  $(\cdot \mid \cdot)_0$ is not positive definite on $\lis$.
  Therefore, $(\cdot \mid \cdot)_0$ is positive definite on $\lih'$.
  Since $U = \lih' \oplus \lin^+ \oplus \lin^-$,
  $(\cdot \mid \cdot)_0$ is positive definite on $U$.
\end{proof}
\begin{lem}
  \label{directDecomp}
  Consider a subspace $V$ of $U$ that is closed
  under the action of $H$.
  Let $V^\perp$ be the subspace of $U$
  orthogonal to $V$ with respect to the
  Hermitian form $(\cdot \mid \cdot)_0$.
  Then $U = V \oplus V^\perp$.
\end{lem}
\begin{proof}
  We consider the eigenspace decomposition of $U$ by $H$.
  Let $U_\lambda$ be the eigenspace for $\lambda$ and write
  \begin{align*}
    U = \bigoplus_{\lambda \in \bbC} U_\lambda.
  \end{align*}
  Since $H$ is a Hermitian operator on $(\cdot \mid \cdot)_0$,
  $U_\lambda$ and $U_\mu$ are orthogonal with respect to this inner product
  if $\lambda \neq \mu$.
  Since $H$ is dominant, $U_\lambda$ is finite-dimensional.
  For each $\lambda$, $V$ also inherits the eigenspace decomposition of $U$.
  Let $V_\lambda$ be an eigenspace of $V$ for $\lambda$, and
  $V$ can be written as a direct sum of $V_\lambda$'s.
  Let
  \begin{align*}
    V'_\lambda = \{ v \in V_\lambda \mid \forall x \in V_\lambda \; (v \mid x)_0 = 0 \},
  \end{align*}
  and
  \begin{align*}
    V' = \bigoplus_{\lambda \in \bbC} V'_\lambda.
  \end{align*}
  $V_\lambda$ is finite dimensional.
  From \cref{positiveDefinite},
  $(\cdot \mid \cdot)_0$ is positive definite on $U$.
  Thus we have $U_\lambda = V_\lambda \oplus V'_\lambda$.
  Therefore, we have $U = V \oplus V'$ and $V' = V^\perp$.
\end{proof}
In the following, we show that $U$ can be decomposed into irreducible modules
by the action of $\lis$.
\begin{lem}
  \label{UhasIrredSubmod}
  Any non-zero $\lisl_2$-submodule $V$ of $U$ includes
  an irreducible submodule.
\end{lem}
\begin{proof}
  Take the eigenspace decomposition of $U$ by the action of $H$.
  $V$ is also decomposed into eigenspaces
  with this decomposition, and each eigenspace of $V$ is finite-dimensional.
  We regard $H$ as a linear transform on $V$
  and take some eigenvalue $\lambda$ of $H$ on $V$.
  Let $U (\lisl_2)$ be an universal enveloping algebra of $\lisl_2$.
  Considering the Casimir element $C$ of $U (\lisl_2)$,
  it preserves $V_\lambda$.
  Since $V_\lambda$ is finite-dimensional, there exists an eigenvector of $C$.
  Let $v$ denote this. Consider the $\lisl_2$-submodule generated by $v$,
  which includes an irreducible submodule.
\end{proof}
\begin{thm}
  \label{decomposabilityOfU}
  (cf. \cite[Theorem 1.2]{kob})
  $U$ can be decomposed into direct sum of irreducible $\lis$-modules,
  and all of these modules are unitarizable.
\end{thm}
\begin{proof}
  We consider a set of irreducible submodules of $U$
  such that these submodules are orthogonal to each other
  with respect to $(\cdot \mid \cdot)_0$.
  Let $T$ be the set.
  We order the elements of $T$ by inclusion.
  Then $T$ is non-empty and inductively ordered.
  Therefore, from Zorn's lemma, $T$ has a maximal element.
  Take a maximal element of $T$ and denote it by $\mcl{M}$.
  Consider the direct sum of all submodules belonging to $\mcl{M}$.
  Let $M$ denote this sum. Suppose $U \neq M$, we derive the contradiction.
  Since $M$ is a subspace of $U$ which is closed by the action of $H$,
  from \cref{directDecomp}, we have $U = M \oplus M^\perp$.
  Since $M^\perp$ is non-zero $\lisl_2$ submodule of $U$,
  from \cref{UhasIrredSubmod}, $M^\perp$ includes an irreducible submodule.
  Let $W$ denote this. we have $\mcl{M} \cup \{ W \} \in T$,
  that is contradict the maximality of $\mcl{M}$.
  Therefore, we have $U = M$, and $U$ can be decomposed into
  direct sum of irreducible submodules.
  Combining this with \cref{positiveDefinite},
  we can also see the unitarizability of the modules.
\end{proof}
\begin{thm}
  \label{decomposabilityOfG}
  $\lig$ can be decomposed into direct sum of irreducible $\lis$-modules,
  which consists $\lis$ itself.
  All of these modules except for $\lis$ are unitarizable. \qed
\end{thm}
\section{$\lisl_2$ modules in $\lig$}
In the following, we consider what kind of modules appear
when $\lig$ is decomposed
into irreducible $\lis$-modules.
In particular, we consider how many unitary principal
or complementary series representations.

For a Lie algebra $\lie{a}$, let $U (\lie{a})$ be the universal enveloping algebra of $\lie{a}$.
Let $V$ be an irreducible $\lis$-module which is an irreducible
component of $\lig$.
The Casimir element $C$ of $U (\lis)$ acts on $V$ by constant multiplication.
Let $\mu$ be this constant.
From \cite[Chapter II, Corollary 1.1.11]{ht},
for an eigenvalue $\lambda_0 \in \bbC$ of $H$ on $V$, some interval $I \subset \bbZ$ exists,
and $V$ can be expressed as a direct sum of 1-dimensional eigenspaces
such that the eigenvalues of $H$ are $\lambda_k = \lambda_0 + 2k$ ($k \in I$).
From \cite[Chapter II, Theorem 1.1.13]{ht},
for an eigenvalue $\lambda$ of $H$ on $V$,
we define $s_1 (k)$ for an integer $k$ as
\begin{align}
  \tag{A} \label{s_1}
  s_1(k) = \frac{8\mu - (\lambda + 2k - 1)^2 + 1}{4}.
\end{align}
We take an element $v_k$ of the eigenspace of $V$
with respect to an eigenvalue $\lambda + 2k$.
Then we have $X (Yv_k) = s_1(k) v_k$.
If $k \in \bbZ$ such that $s_1(k) = 0$ does not exist,
then $V$ is an irreducible module that is neither highest weight module nor lowest weight module.
If there exists a $k \in \bbZ$ such that $s_1(k) = 0$,
$V$ is a highest weight module or a lowest weight module.

Let $\mcl{W}$ be the Weyl group of $\lig$.
Using \cref{LengthOfXIs2}, we may write $H, X, Y$ in $\lis$ as follows.
\begin{align*}
  X &= c_0 w_0 (e_p) + c_1 w_1 (e_q) \hsp (c_0, c_1 \in \bbR, w_0, w_1 \in \mcl{W}, (p, q) \in \{ (0, 1), (0, 0), (1, 1) \}),\\
  Y &= - c_0 w_0 (f_p) - c_1 w_1 (f_q),\\
  H &= - c_0 w_0 (h_p) - c_1 w_1 (h_q).
\end{align*}
Let $\klis, \llis, \mlis, \nlis$ be real numbers such that
$c_0 w_0 (e_p) \in \lig_{\klis \alpha_0 + \llis \alpha_1} ,\; c_1 w_1 (e_q) \in \lig_{\mlis \alpha_0 + \nlis \alpha_1}$.
From \cref{dominantLem}, we can write
$\klis = F_{i + 1} ,\; \llis = F_i ,\; \mlis = F_j ,\; \nlis = F_{j + 1}$
with integers $i, j \geq 0$, and furthermore, $i \in \{ j - 1, j, j + 1 \}$.

When we take the root vector $E \in \lig_{s \alpha_0 + t \alpha_1}$ with $s, t \in \bbZ$,
we want to find out which of the three types of modules $E$ generates
under the action of $\lis$.

We define $L$ in the $xy$-plane as follows.
$L$ is a region satisfying $x \geq 0 ,\; y \geq 0 ,\; (x, y) \neq (0, 0)$,
$x^2 - axy + y^2 \leq 1$ and the following conditions.
\begin{align*}
  x < \klis = F_{i + 1}, \hsp (\text{when } i = j - 1 )\\
  x + y < \klis + \llis = F_i + F_{i + 1}, \hsp (\text{when } i = j)\\
  y < \llis = F_i. \hsp (\text{when } i = j + 1)
\end{align*}
If we take the root $s \alpha_0 + t \alpha_1$ with $s, t \in \bbZ$,
then from \cite[Cor 4.3]{km}, the point in the $xy$-plane given by $(s, t)$
is in the interior or on the boundary of the hyperbola $x^2 - axy + y^2 = 1$.
Let $h_C$ be this hyperbola.
Let $\lambda \in \bbR$ as the value for which $H E = \lambda E$.
We have $\lambda = (s \alpha_0 + t \alpha_1)(H)$.
$\lambda \in (0, 2)$ if and only if $(s, t) \in L$.
In the following, we regard a root $s \alpha_0 + t \alpha_1$ as
a point $(s, t)$ in the $xy$-plane.
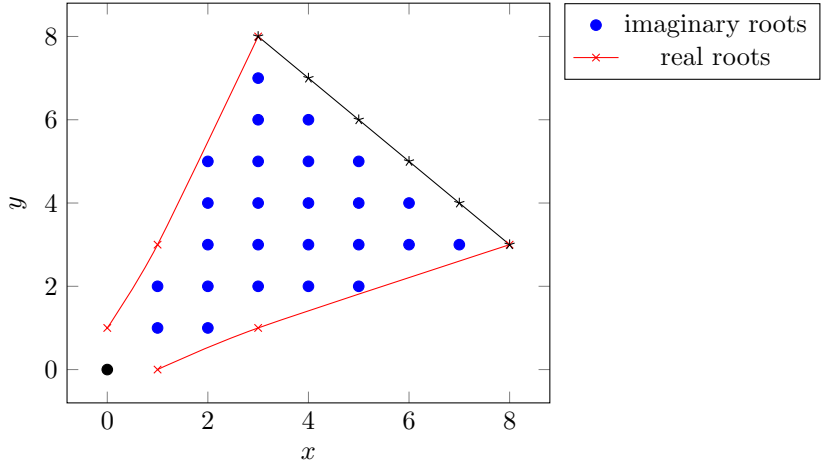
\begin{figure}[H]
  \caption{Imaginary roots and real roots in $L$,
  $a = 3 ,\; X = c_0 r_0 r_1 (e_0) + c_1 r_1 r_0 (e_1)$}
  \label{figure1}
  \pgfplotsset{width=8cm}
  \begin{tikzpicture}
    \begin{axis}[compat = newest,xlabel={$x$},ylabel={$y$},
      legend pos=outer north east]
      \addplot[blue, mark=*, only marks] coordinates{
        (1,1)(1,2)(2,1)(2,2)(2,3)(2,4)(2,5)(3,2)(3,3)(3,4)(3,5)(3,6)(3,7)
        (4,2)(4,3)(4,4)(4,5)(4,6)(5,2)(5,3)(5,4)(5,5)(6,3)(6,4)(7,3)
      };
      \addplot[red, mark=x, smooth] coordinates{(1,0)(3,1)(8,3)};
      \addplot[red, mark=x, smooth] coordinates{(0,1)(1,3)(3,8)};
      \addplot coordinates{(3,8)(4,7)(5,6)(6,5)(7,4)(8,3)};
      \addplot[only marks] coordinates{(0,0)};
      \legend{imaginary roots, real roots}
    \end{axis}
  \end{tikzpicture}
\end{figure}

\begin{lem}
  \label{distanceIsIncreasing}
  We consider the hyperbola $h_C$ on the $xy$-plane. The $h_C$ was represented
  by $x^2 - axy + y^2 = 1$.
  Let $l_b$ be a line represented by the function $y = -x + b$
  with some real number $b \geq 0$.
  There are two intersections of $h_C$ and $l_b$.
  Let $p_1$ and $p_2$ be these points.
  Let $d_b$ be a distance between $p_1$ and $p_2$.
  $d_b$ is strictly monotonically increasing with respect to $b \geq 0$.
  The same result holds when $l_b$ is a line
  represented by $y = b$ or $x = b$.
\end{lem}
\begin{proof}
  First, we consider the case where $l_b$ is represented
  by $y = -x + b$.
  Calculating the $y$-coordinates of $p_1, p_2$ gives
  \begin{align*}
    y = \frac{(a + 2)b \pm \sqrt{(a + 2)(a - 2) b^2 + 4 (a + 2)}}{2(a + 2)}.
  \end{align*}
  Therefore, we have
  \begin{align*}
    d_b = \sqrt{2} \cdot \frac{\sqrt{(a + 2)(a - 2) b^2 + 4 (a + 2)}}{a + 2}.
  \end{align*}
  This $d_b$ is strictly monotonically increasing with respect to $b \geq 0$.

  Next, we consider the case where $l_b$ is represented
  by $y = b$.
  Calculating the $x$-coordinates of $p_1, p_2$ gives
  \begin{align*}
    x = \frac{ab \pm \sqrt{(a^2 - 4) b^2 + 4}}{2}.
  \end{align*}
  Therefore, we have
  \begin{align*}
    d_b = \sqrt{(a^2 - 4) b^2 + 4}.
  \end{align*}
  This $d_b$ is strictly monotonically increasing with respect to $b \geq 0$.
  The same argument is presented when $l_b$ is a line
  represented by $x = b$.
\end{proof}
Let $R$ be the interior of $h_C$ and $h_C$ itself.
For $s, t \in \bbZ$, $(s, t)$ is a root if and only if $(s, t) \in R$.
\begin{lem}
  \label{oneOfDistantPointsIsNotRoot_IEqualsJ}
  If $(x, y) \in L \cup -L$, then neither $(x + \klis - \mlis, y + \llis - \nlis)$
  nor $(x - \klis + \mlis, y - \llis + \nlis)$ are roots.
\end{lem}
\begin{proof}
  First we assume $(x, y) \in L$.
  The points $(\klis, \llis)$ and $(\mlis, \nlis)$ are on the hyperbola $h_C$.
  Let $l_1$ be the line connecting these two points.
  Using some real number $b > 0$,
  $l_1$ is represented by $y = -x + b$ when $i = j$,
  $y = b$ when $i = j - 1$, and $x = b$ when $i = j + 1$.
  Let $l_2$ be a line parallel to $l_1$ and passing
  through $(x, y)$.
  Using some real number $0 < b' < b$,
  $l_2$ is represented by $y = -x + b'$ when $i = j$,
  $y = b'$ when $i = j - 1$, and $x = b'$ when $i = j + 1$.
  Let $p_{11}, p_{12}$ be intersections of $h_C$ and $l_1$.
  Let $d_1$ be the distance between $p_{11}$ and $p_{12}$.
  Let $p_{21}, p_{22}$ be intersections of $h_C$ and $l_2$.
  Let $d_2$ be the distance between $p_{21}$ and $p_{22}$.
  From \cref{distanceIsIncreasing}, we have $d_1 > d_2$.
  The distance between $(\klis, \llis)$ and $(\mlis, \nlis)$ is $d_1$.
  The distance between $(x, y)$ and $(x + \klis - \mlis, y + \llis - \nlis)$
  is also $d_1$. These two points are on $l_2$.
  The length of the part of $l_2$ that is inside the hyperbola
  is $d_2 < d_1$.
  From the fact that $(x, y)$ is inside $h_C$,
  $(x + \klis - \mlis, y + \llis - \nlis)$ is outside the hyperbola.
  Therefore, $(x + \klis - \mlis, y + \llis - \nlis)$ is not in $R$.
  The same argument for $(x - \klis + \mlis, y - \llis + \nlis)$
  shows that it is not in $R$.
  From symmetry, the case when $(x, y) \in -L$ is also shown.
\end{proof}
\begin{lem}
  \label{XYeIsConstant}
  For a point $(s, t) \in L$ corresponding to the root,
  we consider the root vector $E \in \lig_{s \alpha_0 + t \alpha_1}$.
  Then $[X, [Y, E]] \in \lig_{s \alpha_0 + t \alpha_1}$.
\end{lem}
\begin{proof}
  We have
  $X = c_0 w_0 (e_p) + c_1 w_1 (e_q), Y = - c_0 w_0 (f_p) - c_1 w_1 (f_q)$.
  Also we have
  $c_0 w_0 (e_p) \in \lig_{\klis \alpha_0 + \llis \alpha_1} ,\; c_1 w_1 (e_q) \in \lig_{\mlis \alpha_0 + \nlis \alpha_1}$.
  Then we have
  \begin{align*}
    [X, [Y, E]] \in \lig_{s \alpha_0 + t \alpha_1} + \lig_{(s - \klis + \mlis) \alpha_0 + (t - \llis + \nlis) \alpha_1}
    + \lig_{(s - \mlis + \klis) \alpha_0 + (t - \nlis + \llis) \alpha_1}.
  \end{align*}
  Since $(s, t)$ is a root,
  from \cref{oneOfDistantPointsIsNotRoot_IEqualsJ},
  $(s - \klis + \mlis, t - \llis + \nlis)$ and $(s - \mlis + \klis, t - \nlis + \llis)$ are not roots.
  Therefore, we have $\lig_{(s - \klis + \mlis) \alpha_0 + (t - \llis + \nlis) \alpha_1} + \lig_{(s - \mlis + \klis) \alpha_0 + (t - \nlis + \llis) \alpha_1} = 0$, and $[X, [Y, E]] \in \lig_{s \alpha_0 + t \alpha_1}$.
\end{proof}
We consider the Casimir element $C$ of $U (\lis)$.
We can write $C = \frac{1}{8} H^2 - \frac{1}{4} H + \frac{1}{2} XY$.
\begin{lem}
  \label{CasimirIsDiagonalizable}
  We consider a root space with respect to a root in $L$.
  $C$ acts on the root space as endomorphism.
  The action is diagonalizable.
\end{lem}
\begin{proof}
  From \cref{XYeIsConstant}, $C$ acts on the root spaces as endomorphism.
  Since $\lig$ is completely reducible as an $\lis$-modules,
  the action on the root space is diagonalizable.
\end{proof}
\begin{lem}
  \label{EgeneratesIrreducibleModule}
  For a point $(s, t) \in L$ corresponding to the root,
  we can take the root vector $E \in \lig_{s \alpha_0 + t \alpha_1}$ such that
  $E$ is an eigenvector of the Casimir element $C$, and
  $E$ generates an irreducible $\lis$-module.
\end{lem}
\begin{proof}
  From \cref{CasimirIsDiagonalizable}, we have the lemma.
\end{proof}
From \cref{EgeneratesIrreducibleModule}, if we decompose $\lig$ by the action of $\lis$,
the decomposition is compatible with the root space decomposition
in the root in $L$.

We consider how many unitary principal or complementary series representations appear
in the decomposition of $\lig$.
Since the set of eigenvalues of unitary principal or complementary series representations
is $\{ \lambda + 2k \mid k \in \bbZ \}$ for some $\lambda$,
such a module must contain an eigenspace such that its eigenvalue lie
on $[0, 2)$.
Therefore, we consider the root vector of $H$ such that the eigenvalue $\lambda$
of $H$ satisfies $\lambda \in [0, 2)$.

If $\lambda = 0$, i.e., $s = t = 0$,
Since the dimension of $\lih$ is 2,
there are two irreducible components of $V$ which have
0-eigenspace (cf. \cite[\S 7]{tsu}).
Since one is $\lisl_2$ itself, we consider the other module.
The casimir element $C$ acts on this module by a constant multiple (let $\mu$ times).
If $k$ satisfies $s_1 (k) = 0$, we get $8 \mu + 1 = ( 2k - 1)^2$.
Since $\mu < -1$ from \cite[Proposition 7.3]{tsu},
the left hand side is less than 0. Therefore, there is no integral solution
to $s_1 (k) = 0$, and this is an irreducible module that is
neither highest weight module nor lowest weight module.
In particular, this module is an unitary principal series representation.

In the following, we consider the case of $\lambda \in (0, 2)$.
In this case, $(s, t)$ is a root in $L$.
We compute $[X, [Y, E]]$.
Since $Y = -c_0 w_0 (f_p) - c_1 w_1 (f_q)$, we have
\begin{align*}
  [Y, E] = [-c_0 w_0 (f_p), E] + [-c_1 w_1 (f_q), E].
\end{align*}
We have also
\begin{align*}
  [-c_0 w_0 (f_p), E] \in \lig_{(s - \klis)\alpha_0 + (t - \llis)\alpha_1},\\
  [-c_1 w_1 (f_q), E] \in \lig_{(s - \mlis)\alpha_0 + (t - \nlis)\alpha_1}.
\end{align*}
If $[-c_0 w_0 (f_p), E]$ and $[-c_1 w_1 (f_q), E]$ are not 0, then the eigenvalue of $H$ for them must be
in the $(-2, 0)$ interval.
We consider root vectors which the eigenvalue of $H$ are in the $(-2, 0)$.
Since $R = -R$,
the roots with respect to these root vectors are $-L$.
From \cref{oneOfDistantPointsIsNotRoot_IEqualsJ},
if we take two points such that the difference is $(\klis - \mlis, \llis - \nlis)$
and one of which is a root in $-L$, then the other is not a root.
Now we have $((s - \mlis) - (s - \klis), (t - \nlis) - (t - \llis)) = (\klis - \mlis, \llis - \nlis)$.
Therefore, we know that at least one of $[-c_0 w_0 (f_p), E], [-c_1 w_1 (f_q), E]$
is zero.

When both of these are 0, we have $[Y, E] = 0$ and
from the fact that $C = \frac{1}{8} H^2 - \frac{1}{4} H + \frac{1}{2} XY$,
we can write $8 \mu = \lambda^2 - 2\lambda$.

When $[-c_0 w_0 (f_p), E] \neq 0$, i.e., $(s - \klis, t - \llis) \in R$,
we have
\begin{align*}
  [X, [Y, E]] &= [c_0 w_0 (e_p), [-c_0 w_0 (f_p), E]]\\
  &= [E, [-c_0 w_0 (f_p), c_0 w_0 (e_p)]]
  + [-c_0 w_0 (f_p), [c_0 w_0 (e_p), E]].
\end{align*}
We define $\plis \in \bbC$ by $[-c_0 w_0 (f_p), [c_0 w_0 (e_p), E]] = \plis E$,
then we have
\begin{align*}
  [X, [Y, E]] = [E, c_0^2 w_0 (h_p)] + \plis E.
\end{align*}
When $\plis = 0$, we have
\begin{align*}
  [X, [Y, E]] = - [c_0^2 w_0 (h_p), E].
\end{align*}
Therefore in this case, if we let $- [c_0^2 w_0 (h_p), E] = k_0 E$,
then we have $8 \mu = \lambda^2 - 2\lambda + 4k_0$.

When $[c_0 w_0 (e_p), E] = 0$, i.e., $(s + \klis, t + \llis) \not\in R$,
we have $\plis = 0$.

To summarize the above, we take an irreducible decomposition of $\lig$ by $\lis$.
let $s \alpha_0 + t \alpha_1$ be a root in $L$.
Let $E \in \lig_{s \alpha_0 + t \alpha_1}$ such that $E$ generates an irreducible component
of $\lig$.
Let $C$ be the Casimir element of $U (\lis)$, and Let $\mu$ be a complex number
such that $CE = \mu E$.
Let $k_0$ and $p_\lis$ be complex numbers satisfying
\begin{align*}
  [- c_0^2 w_0 (h_p), E] &= k_0 E,\\
  [-c_0 w_0 (f_p), [c_0 w_0 (e_p), E]] &= \plis E.
\end{align*}
If $(s - \mlis, t - \nlis) \not\in R$, we have
\begin{align*}
  8 \mu = \left\{
    \begin{aligned}
      &\lambda^2 - 2\lambda &\left( (s - \klis, t - \llis) \not\in R \right),\\
      &\lambda^2 - 2\lambda + 4k_0
      &\left( (s - \klis, t - \llis) \in R \text{ and } (s + \klis, t + \llis) \not\in R \right),\\
      &\lambda^2 - 2\lambda + 4k_0 + \plis
      &\left( (s - \klis, t - \llis) \in R \text{ and } (s + \klis, t + \llis) \in R \right).
    \end{aligned}
  \right.
\end{align*}
If $(s - \klis, t - \llis) \not\in R$ and not necessarily $(s - \mlis, t - \nlis) \not\in R$,
we have
\begin{align*}
  8 \mu = \left\{
    \begin{aligned}
      &\lambda^2 - 2\lambda &\left( (s - \mlis, t - \nlis) \not\in R \right),\\
      &\lambda^2 - 2\lambda + 4k_0
      &\left( (s - \mlis, t - \nlis) \in R \text{ and } (s + \mlis, t + \nlis) \not\in R \right),\\
      &\lambda^2 - 2\lambda + 4k_0 + \plis
      &\left( (s - \mlis, t - \nlis) \in R \text{ and } (s + \mlis, t + \nlis) \in R \right).
    \end{aligned}
  \right.
\end{align*}
Solving
\begin{align*}
  s_1(k) = \frac{8\mu - (\lambda + 2k - 1)^2 + 1}{4} = 0
\end{align*}
for $k$ on $\bbR$,
we obtain that
\begin{align*}
  k = \left\{
    \begin{aligned}
      & 0, 1 - \lambda &\left( (s - \klis, t - \llis) \not\in R \text{ and } (s - \mlis, t - \nlis) \not\in R \right),\\
      & \frac{1 - \lambda \pm \sqrt{(\lambda - 1)^2 + 4 k_0}}{2}
      &\left(
        \begin{aligned}
          (s - \klis, t - \llis) &\in R \text{ and } (s + \klis, t + \llis) \not\in R\\
          \text{or}\\
          (s - \mlis, t - \nlis) &\in R \text{ and } (s + \mlis, t + \nlis) \not\in R
        \end{aligned}
      \right),\\
      & \frac{1 - \lambda \pm \sqrt{(\lambda - 1)^2 + 4 k_0 + \plis}}{2}
      &\left(
        \begin{aligned}
          (s - \klis, t - \llis) &\in R \text{ and } (s + \klis, t + \llis) \in R\\
          \text{or}\\
          (s - \mlis, t - \nlis) &\in R \text{ and } (s + \mlis, t + \nlis) \in R
        \end{aligned}
      \right).
    \end{aligned}
  \right.
\end{align*}
When $(s - \klis, t - \llis) \not\in R$ and $(s - \mlis, t - \nlis) \not\in R$,
since $(s, t) \in L$, we have $1 - \lambda \in (-1, 1)$.
Therefore, we know that the only integral solution of $s_1 (k) = 0$ is 0.
In this case $E$ belongs to an irreducible lowest weight module.
\section{Classification by roots}
Based on the previous section, we classify the root $(s, t)$ in $L$.
We define the types of roots as follows.
\begin{newEnum}
  \item We say that $(s, t)$ is of type A when $(s - \klis, t - \llis) \not\in R \text{ and } (s - \mlis, t - \nlis) \not\in R$.
  \item We say that $(s, t)$ is of type B when $\left\{
    \begin{aligned}
      (s - \klis, t - \llis) &\in R \text{ and } (s + \klis, t + \llis) \not\in R\\
      \text{or}\\
      (s - \mlis, t - \nlis) &\in R \text{ and } (s + \mlis, t + \nlis) \not\in R
    \end{aligned}
  \right\}$.
  \item We say that $(s, t)$ is of type C when $\left\{
    \begin{aligned}
      (s - \klis, t - \llis) &\in R \text{ and } (s + \klis, t + \llis) \in R\\
      \text{or}\\
      (s - \mlis, t - \nlis) &\in R \text{ and } (s + \mlis, t + \nlis) \in R
    \end{aligned}
    \right\}$.
\end{newEnum}
All roots belong to one of the above types.
We put $f(x, y) = x^2 - axy + y^2$ for $x, y \in \bbR$.
From \cite[Cor 4.3]{km}, for $s, t \in \bbZ,\; (s, t) \neq (0, 0)$,
$(s, t)$ is a real root if and only if $f(s, t) = 1$,
and $(s, t)$ is an imaginary root if and only if $f(s, t) < 1$.
\begin{lem}
  \label{fIsInvariantOfReflection}
  For $x, y, x', y' \in \bbR$, if there exists $w \in \mcl{W}$
  such that $(x', y') = w(x, y)$, then $f(x', y') = f(x, y)$.
\end{lem}
\begin{proof}
  It is sufficient to check the case $w = r_0$ and the case $w = r_1$.
  From the symmetry, it is sufficient to check the case $w = r_0$.
  In this case, from the fact that $x' = ay - x$ and $y' = y$,
  we have
  \begin{align*}
    f(x', y') &= f(ay - x, y)\\
    &= (ay - x)^2 - ay(ay - x) + y^2\\
    &= x^2 - axy + y^2\\
    &= f(x, y).
  \end{align*}
\end{proof}
First, we know the following results on real roots.
\begin{lem}
  \label{realRootDifferenceIsNegative}
  If $(s, t)$ is a real root in $L$ and $s > t$,
  then $f (s - \klis, t - \llis) \leq 0$.
  Also, if $(s, t)$ is a real root in $L$ and $s < t$,
  then $f (s - \mlis, t - \nlis) \leq 0$.
\end{lem}
\begin{proof}
  From symmetry, it is sufficient to show $f (s - \klis, t - \llis) \leq 0$ when $s > t$.
  We can write $s = F_{c + 1} ,\; t = F_c$
  with $c \geq 0$ being an integer.
  Since $\klis = F_{i + 1}$ and $\llis = F_i$,
  we have $c < i$.
  Let $d_{ic} = i - c$.
  From \cref{fIsInvariantOfReflection}, by acting $r_0$ and $r_1$ on $(s - \klis, t - \llis)$,
  we know that
  \begin{align*}
    f(s - \klis, t - \llis) &= f(F_{c + 1} - F_{i + 1}, F_c - F_i) \\
    &= f (r_0 (F_{c + 1} - F_{i + 1}, F_c - F_i))\\
    &= f (F_{c - 1} - F_{i - 1}, F_c - F_i)\\
    &= f (r_1 (F_{c - 1} - F_{i - 1}, F_c - F_i))\\
    &= f (F_{c - 1} - F_{i - 1}, F_{c - 2} - F_{i - 2})\\
    &= \cdots\\
    &= \left\{
    \begin{aligned}
      f (F_1 - F_{d_{ic} + 1}, F_0 - F_{d_{ic}}) \hsp (\text{when $c$ is even})\\
      f (F_0 - F_{d_{ic}}, F_1 - F_{d_{ic} + 1}) \hsp (\text{when $c$ is odd})
    \end{aligned}
    \right.\\
    &= f(F_1 - F_{d_{ic} + 1}, F_0 - F_{d_{ic}}).
  \end{align*}
  Since $F_1 = 1 ,\; F_0 = 0$, we have
  \begin{align*}
    f(s - \klis, t - \llis)
    &= f(1 - F_{d_{ic} + 1}, - F_{d_{ic}})\\
    &= 2 - a F_{d_{ic}} + 2 F_{d_{ic} - 1}\\
    &< 2 - 2(F_{d_{ic}} - F_{d_{ic} - 1})\\
    &\leq 0.
  \end{align*}
\end{proof}
\begin{lem}
  \label{realRootIsTypeB}
  If $(s, t)$ is a real root in $L$,
  then $(s, t)$ is of type B.
\end{lem}
\begin{proof}
  First we show that $(s, t)$ is not of type A.
  From the fact that $(s, t)$ is a real root and
  from symmetry, we can write $s = F_{c + 1} ,\; t = F_c$
  with $c \geq 0$ being an integer.
  From $\klis = F_{i + 1} ,\; \llis = F_{i}$,
  we have $c < i$.
  From \cref{realRootDifferenceIsNegative},
  $f (s - \klis, t - \llis) \leq 0$.
  Therefore, $(s - \klis, t - \llis) \in R$ and so we know that $(s, t)$ is not of type A.

  Next, we show that $(s, t)$ is of type B.
  To show this, we need to show that $(s + \klis, t + \llis) \not\in R$.
  We show $f(s + \klis, t + \llis) > 1$.
  Let $d_{ic} = i - c$.
  From \cref{fIsInvariantOfReflection}, by acting $r_0$ and $r_1$ on $(s + \klis, t + \llis)$,
  we have
  \begin{align*}
    f(s + \klis, t + \llis)
    &= f(F_{c + 1} + F_{i + 1}, F_c + F_i)\\
    &= f (r_0 (F_{c + 1} + F_{i + 1}, F_c + F_i))\\
    &= f (F_{c - 1} + F_{i - 1}, F_c + F_i)\\
    &= f (r_1 (F_{c - 1} + F_{i - 1}, F_c + F_i))\\
    &= f (F_{c - 1} + F_{i - 1}, F_{c - 2} + F_{i - 2})\\
    &= \cdots\\
    &= \left\{
    \begin{aligned}
      f (F_1 + F_{d_{ic} + 1}, F_0 + F_{d_{ic}}) \hsp (\text{when $c$ is even})\\
      f (F_0 + F_{d_{ic}}, F_1 + F_{d_{ic} + 1}) \hsp (\text{when $c$ is odd})
    \end{aligned}
    \right.\\
    &= f (F_1 + F_{d_{ic} + 1}, F_0 + F_{d_{ic}})\\
    &= f(1 + F_{d_{ic} + 1}, F_{d_{ic}})\\
    &= 2 + a F_{d_{ic}} - 2 F_{d_{ic} - 1}\\
    &> 2 + 2(F_{d_{ic}} - F_{d_{ic} - 1})\\
    &> 4.
  \end{align*}
  This shows that $(s, t)$ is of type B.
\end{proof}
We classify also for imaginary roots in $L$.
\begin{lem}
  \label{imagPlusImagisImag}
  If $(s, t), (s', t')$ ($s, t, s', t' \geq 0$) are imaginary roots,
  then $(s + s', t + t')$ is also imaginary root.
\end{lem}
\begin{proof}
  Since $f(s, t) \leq 0$, for any $r \in \bbR$,
  we have $f(rs, rt) = r^2 f(s, t) \leq 0$.
  It shows that the line connecting the origin and $(s, t)$
  is inside the asymptotes of the hyperbola $x^2 - axy + y^2 = 1$.
  Similarly, the line connecting the origin and $(s', t')$
  is also inside the asymptotes.

  Since $(s + s', t + t')$ is the midpoint of $(2s, 2t)$ and $(2s', 2t')$,
  this point is also inside the asymptotes. Therefore,
  $(s + s', t + t')$ is an imaginary root.
\end{proof}
\begin{lem}
  \label{imagRootIsTypeAorC}
  Let $(u, v) \in L$ ($u > v$) be a real root such that $(u \alpha_0 + v \alpha_1)(H) \neq 0$.
  Put $(s, t) = (\klis - u, \llis - v)$.
  Then $(s, t)$ is a type C imaginary root in $L$.
  Similarly, let $(u', v') \in L$ ($u' < v'$) be a real root such that
  $(u' \alpha_0 + v' \alpha_1)(H) \neq 0$.
  Put $(s', t') = (\mlis - u', \nlis - v')$.
  Then $(s', t') \in L$ and $(s', t')$ is the imaginary root of type C.
  The other imaginary roots in $L$ are of type A.
\end{lem}
\begin{proof}
  From \cref{realRootDifferenceIsNegative},
  $f(-s, -t) = f(s, t) \leq 0$. It shows that $(s, t)$ is a imaginary root.
  We also see that the eigenvalue of $H$ for $(s, t)$ is
  in the range $(0, 2)$. Therefore, $(s, t) \in L$ is shown.

  We show that $(s, t)$ is of type C. To show this,
  we show that $f(s + \klis, t + \llis) \leq 1$.
  Using $c \in \bbZ$, we can write $(u, v) = (F_{c + 1}, F_c)$.
  Together this with $s + \klis = 2 \klis - u ,\; t + \llis = 2 \llis - v$,
  we have
  \begin{align*}
    f(s + \klis, t + \llis) = f(2F_{i + 1} - F_{c + 1}, 2F_i - F_c).
  \end{align*}
  Let $d_{ic} = i - c > 0$.
  From \cref{fIsInvariantOfReflection},
  acting $r_0, r_1$ on $(s + \klis, t + \llis)$,
  $i - c = \lambda \geq 1$, we have
  \begin{align*}
    f(2F_{i + 1} - F_{c + 1}, 2F_i - F_c)
    &= f (r_0 (2F_{i + 1} - F_{c + 1}, 2F_i - F_c))\\
    &= f (2F_{i - 1} - F_{c - 1}, 2F_i - F_c)\\
    &= f (r_1 (2F_{i - 1} - F_{c - 1}, 2F_i - F_c))\\
    &= f (2F_{i - 1} - F_{c - 1}, 2F_{i - 2} - F_{c - 2})\\
    &= \cdots\\
    &= \left\{
      \begin{aligned}
        f (2F_{d_{ic} + 1} - F_1, 2F_{d_{ic}} - F_0) \hsp (\text{when $c$ is even})\\
        f (2F_{d_{ic}} - F_0, 2F_{d_{ic} + 1} - F_1) \hsp (\text{when $c$ is odd})
      \end{aligned}
      \right.\\
    &= f(2F_{d_{ic} + 1} - F_1, 2F_{d_{ic}} - F_0)\\
    &= f(2F_{d_{ic} + 1} - 1, 2F_{d_{ic}})\\
    &= - 2a F_{d_{ic}} + 4F_{d_{ic} - 1} + 5\\
    &< -6 F_{d_{ic}} + 4 F_{d_{ic} - 1} + 5\\
    &= (-4 F_{d_{ic}} + 4 F_{d_{ic} - 1})
    + (-2 F_{d_{ic}} + 5)\\
    &< -4 - 2F_{d_{ic}} + 5\\
    &\leq -1.
  \end{align*}
  This shows that $f(s + \klis, t + \llis) \leq -1$
  and that $(s, t)$ is type C.
  From symmetry, we also know that $(s', t')$ is in $L$
  and is the imaginary root of type C.

  Finally, we show the other imaginary roots in $L$ are of type A.
  Let $(s'', t'') \in L$ be such an imaginary root.
  We show $(s'' - \mlis, t'' - \nlis) \not\in R$ and $(s'' - \klis, t'' - \llis) \not\in R$.
  If $(s'' - \mlis, t'' - \nlis) \in R$ or $(s'' - \klis, t'' - \llis) \in R$,
  $(s'' - \mlis, t'' - \nlis) \in -L$ or $(s'' - \klis, t'' - \llis) \in -L$.
  Since $(s'' - \mlis, t'' - \nlis) - (s'' - \klis, t'' - \llis) = (\klis - \mlis, \llis - \nlis)$,
  from \cref{oneOfDistantPointsIsNotRoot_IEqualsJ},
  we know $(s'' - \mlis, t'' - \nlis) \not\in R$ or $(s'' - \klis, t'' - \llis) \not\in R$.

  From symmetry, it is sufficient to consider when
  $(s'' - \mlis, t'' - \nlis) \not\in R$.
  Under this assumption, $(s'' - \klis, t'' - \llis)$
  is an imaginary root or not a root.
  If $(s'' - \klis, t'' - \llis)$ is imaginary root,
  then $(\klis - s'', \llis - t'')$ is also imaginary root
  from the symmetry of $R$.
  We consider that $(\klis, \llis) = (s'', t'') + (\klis - s'', \llis - t'')$.
  The left hand side is real root and the right hand side is the sum
  of imaginary roots, which contradicts \cref{imagPlusImagisImag}.
  Therefore, $(s'' - \klis, t'' - \llis)$ is not a root and
  $(s'', t'')$ is of type A.
\end{proof}
The contents of this section can be summarized as follows.
\begin{thm}
  \label{classifyRootsByType}
  \begin{newEnum}
    \item A real roots in $L$ is of type B.
    \item We consider an imaginary root that can be
    written as $(\klis - s, \llis - t)$ or
    $(\mlis - s, \nlis - t)$ where $(s, t)$ is a real root. Such an imaginary root is of type C.
    \item The other imaginary roots are of type A.\qed
  \end{newEnum}
\end{thm}
We now summarize the irreducible $\lis$-modules
through type A and type C.
For $\lis$-modules through type A, we have the following.
\begin{lem}
  \label{typeAIsLowest}
  An irreducible $\lis$-module containing a root vector about
  a root of type A in $L$ is a lowest weight module
  which the root vector is the lowest weight element.
\end{lem}
\begin{proof}
  Since $(s - \klis, t - \llis) \not\in R$ and
  $(s - \mlis, t - \nlis) \not\in R$ for the root $(s, t)$
  of type A, we know that acting $Y$ on the type A root
  vector will result in 0. This shows the lemma.
\end{proof}
\begin{lem}
  \label{typeCIsLowestOrTypeB}
  Let $M$ be an irreducible $\lis$-module containing a root vector (say $v$)
  with respect to type C root in $L$.
  Then one of the following conditions \cref{typeCItem1} or \cref{typeCItem2} is hold.
  \begin{newEnum}
    \item \label{typeCItem1} $M$ is a lowest weight module such that $v$ is a lowest element.
    \item \label{typeCItem2} $M$ contains a real root vector with respect to a real root in $-L$.
  \end{newEnum}
\end{lem}
\begin{proof}
  The type C root $(s, t)$ can be written with some real root
  $(s_r, t_r)$ that
  $(\klis - s_r, \llis - t_r)$ or $(\mlis - s_r, \nlis - t_r)$.
  Therefore, the root vector $E$ of type C
  becomes either zero or a real root vector when $Y$ act on it.
  If $E$ becomes 0 under the action of $Y$,
  then $E$ generates an irreducible lowest weight module.
  If $E$ becomes a real root vector,
  then the real root for this vector is in $-L$,
  and this lemma is shown.
\end{proof}
We also give the type A, B, C distinction to the root of $-L$
by defining \cref{classifyRootsByType}.
Then, if there is a unitary principal or complementary series representation
that passes through a root vector of type C in $L, -L$,
it will also pass through the root vector of type B in $-L, L$.
Therefore, We have only to classify the modules
that contains a type B root space.
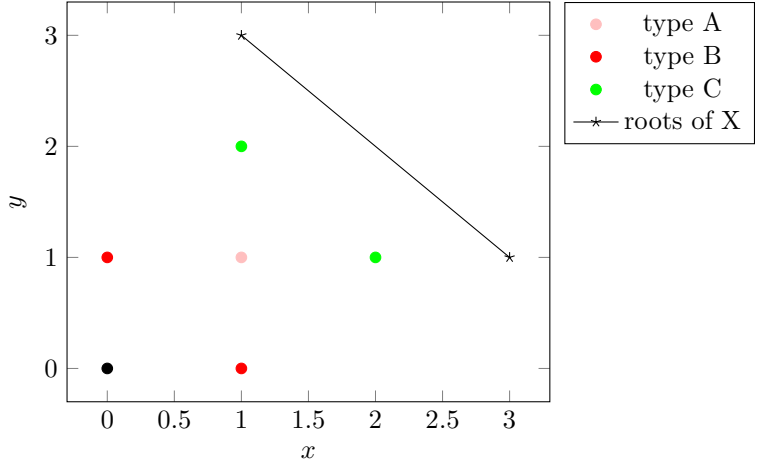
\begin{figure}[H]
  \caption{$a = 3 ,\; X = c_0 r_0 (e_1) + c_1 r_1 (e_0)$}
  \pgfplotsset{width=8cm}
  \begin{tikzpicture}
    \begin{axis}[compat = newest,xlabel={$x$},ylabel={$y$},
      legend pos=outer north east]
      \addplot[pink, only marks] coordinates{(1,1)};
      \addplot[red, only marks] coordinates{(1,0)(0,1)};
      \addplot[green, only marks] coordinates{(2,1)(1,2)};
      \addplot coordinates{(3,1)(1,3)};
      \addplot[only marks] coordinates{(0,0)};
      \legend{type A, type B, type C, roots of X}
    \end{axis}
  \end{tikzpicture}
\end{figure}

\begin{figure}[H]
  \caption{$a = 3 ,\; X = c_0 r_0 (e_1) + c_1 r_1 r_0 (e_1)$}
  \pgfplotsset{width=8cm}
  \begin{tikzpicture}
    \begin{axis}[compat = newest,xlabel={$x$},ylabel={$y$},
      legend pos=outer north east]
      \addplot[pink, only marks] coordinates{(1,1)(1,2)(2,2)(2,3)(2,4)};
      \addplot[red, only marks] coordinates{(1,0)(0,1)(1,3)};
      \addplot[green, only marks] coordinates{(2,1)(2,5)};
      \addplot coordinates{(3,1)(3,8)};
      \addplot[only marks] coordinates{(0,0)};
      \legend{type A, type B, type C, roots of X}
    \end{axis}
  \end{tikzpicture}
\end{figure}
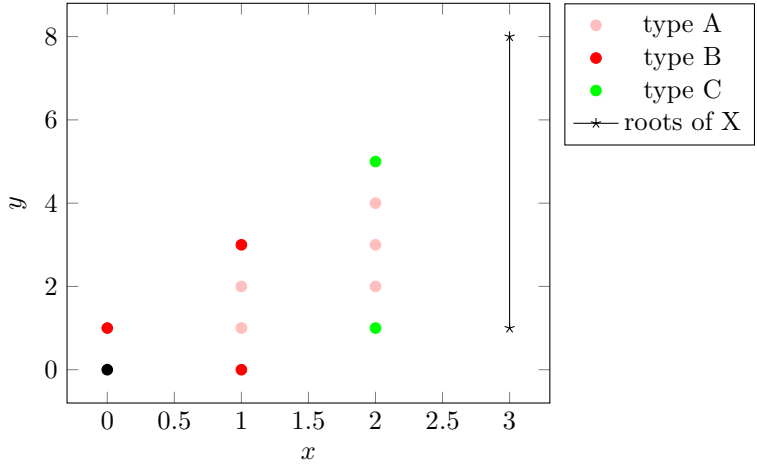
\section{Irreducible modules which contains a root space with respect to a type B root}
We consider an irreducible decomposition of $\lig$ by $\lis$, and
we consider an irreducible component $M$ containing a type B root space.
The multiplicity of a real root space is 1.
We can take $0 < \lambda < 2$ such that $\{ \lambda + 2k' \mid k' \in \bbZ \}$
is the set of the eigenvalues of $H$ in $M$.
We consider the $H$ eigenspace of $M$ such that the eigenvalue is $\lambda$.
We assume this eigenspace is $\lig_{s \alpha_0 + t \alpha_1}$ such that $(s, t) \in L$,
and $(s, t)$ is real root.
We consider $k$ such that $s_1 (k) = 0$ in \cref{s_1} in \S 5.
We show that it is not an integer.

Let $e_0, e_1, f_0, f_1, h_0$ and $h_1$ be Chevalley generators.
Using some $c_0, c_1 \in \bbR, w_0, w_1 \in \mcl{W}$, and $(p, q) \in \{ (0, 1), (0, 0), (1, 1) \}$,
let $X = c_0 w_0 (e_p) + c_1 w_1 (e_q)$.
Suppose $s > t$.
We take the root vector $E$ with respect to the root $s \alpha_0 + t \alpha_1$.
We define $\lambda$ by $H E = \lambda E$,
and define $k_0$ by $[- c_0^2 w_0 (h_p), E] = k_0 E$.
Thus $s_1 (k) = 0$ implies
\begin{align*}
  k = \frac{1 - \lambda \pm \sqrt{(\lambda - 1)^2 + 4 k_0}}{2}.
\end{align*}
We put
\begin{align*}
  k_+ &= \frac{1 - \lambda + \sqrt{(\lambda - 1)^2 + 4 k_0}}{2},\\
  k_- &= \frac{1 - \lambda - \sqrt{(\lambda - 1)^2 + 4 k_0}}{2}
\end{align*}
and we show that $k_\pm \not\in \bbR$ or $0 < k_\pm < 1$.

From \cref{dominantLem}, we can write
$c_0 w_0 (e_p) \in \lig_{F_{i + 1} \alpha_0 + F_i \alpha_1} ,\; (s, t) = (F_{i' + 1}, F_{i'})$ such that
$0 \leq i' < i$. $F_k$ is strictly increasing with respect to $k$. Since $F_{k + 2} = a F_{k + 1} - F_{k}$ and $a \geq 3$,
we have $F_{k} > 2 F_{k - 1}$. Since
\begin{align*}
  \lambda = \left\{
    \begin{aligned}
      \frac{F_{i' + 1}}{F_{i + 1}} \hsp (\text{when } i = j - 1)\\
      \frac{F_{i' + 1} + F_{i'}}{F_{i + 1} + F_i} \hsp (\text{when } i = j)\\
      \frac{F_{i'}}{F_i} \hsp (\text{when } i = j + 1)
    \end{aligned}
  \right\},
\end{align*}
we have $0 < \lambda < 1$.
When $(\lambda - 1)^2 + 4 k_0 < 0$ or $(\lambda - 1)^2 + 4 k_0 \not\in \bbR$, $k_\pm$ are imaginary numbers.
Therefore we can assume $(\lambda - 1)^2 + 4 k_0 \geq 0$.
From $0 < \lambda < 1$, it is clear that $k_+ > 0$ and $k_- < 1$.
To show $k_+ < 1$, we need to show
\begin{align*}
  1 - \lambda + \sqrt{(\lambda - 1)^2 + 4 k_0} < 2.
\end{align*}
we can easily show that it is reduced to $k_0 < \lambda$.
Also, to show that $k_- > 0$,
we need to show
\begin{align*}
  1 - \lambda - \sqrt{(\lambda - 1)^2 + 4 k_0} > 0.
\end{align*}
we can easily show that it is reduced to $k_0 < 0$.
In summary, we have only to show that $k_0 < 0$.

First, consider the case $(s, t) = (1, 0)$, i.e.,
$E \in \lig_{\alpha_0}$. Since
\begin{align*}
  k_0 E &= [- c_0^2 r_0 r_1 r_0 \ldots r_{1 - p} (h_p), E]\\
  &= [- c_0^2 (F_{i + 1} h_0 + F_i h_1), E]\\
  &= - c_0^2 (2 F_{i + 1} - a F_i) E\\
  &= - c_0^2 (F_{i + 1} + F_{i - 1}) E,
\end{align*}
we have $k_0 < 0$.
When $(s, t) = (0, 1)$, we can show that $k_0 < 0$
by replacing $i$ with $j$, $p$ with $q$ and making the same argument.

If $(s, t)$ is general and $s > t$, we can write $(s, t) = (F_{i' + 1}, F_{i'})$.
Let $p'$ be 0 or 1, we can write $E = r_0 r_1 r_0 \ldots r_{1 - p'} (e_{p'})$.
From this, we have
\begin{align*}
  [- c_0^2 w_0 (h_p), E] &= - c_0^2 [r_0 r_1 r_0 \ldots r_{1 - p} (h_p), r_0 r_1 r_0 \ldots r_{1 - p'} (e_{p'})]\\
  &= - c_0^2 r_0 r_1 r_0 \ldots  r_{1 - p'} [r_{p'} r_{1 - p'} r_{p'} \ldots r_{1 - p} (h_p), e_{p'}].
\end{align*}
We consider $k_0$ and $c_0$ when
$i$ is replaced by $i - i'$, and rewrite them as $k_0'$ and $c_0'$.
Considering $(s, t) = (1, 0)$ or $(0, 1)$ cases,  we have $[r_{p'} r_{1 - p'} r_{p'} \ldots r_{1 - p} (h_p), e_{p'}] = - \frac{k_0'}{c_0'} e_{p'}$.
That is, $k_0 = \frac{c_0^2}{c_0'^2} k_0'$.
Since $k_0' < 0$, we have $k_0 < 0$. When $s < t$, we can show that $k_0 < 0$ as well.

From the above, it can be shown that $k_0 < 0$ in any case,
i.e., $k$ is not an integer. From this, we can see the following.
\begin{thm}
  \label{classification}
  We consider an irreducible decomposition of $\lig$ by the action of $\lis$.
  \begin{newEnum}
    \item \label{classification1} Let $M$ is an irreducible component of decomposition of $\lig$, which
    contain a root space for a type B root $s \alpha_0 + t \alpha_1$.
    Then, $M$ is an unitary principal or complementary series representation.
    \item \label{classification2} (cf. \cite[Proposition 7.3]{tsu}) There is an unitary principal
    series representation containing an 1-dimensional space in $\lih$.
    \item $\lig$ is decomposed into a direct sum of $\lis$-submodules described in \cref{classification1}
    and \cref{classification2} above, $\lis$ itself, irreducible lowest weight modules,
    and irreducible highest weight modules. \qed
  \end{newEnum}
\end{thm}
From \cite[\S 3]{km}, the multiplicity of each root of $\lig$ is calculated.
Using this, we can find how many modules appear such that the following condition is satisfied:
the modules are highest or lowest modules, and eigenvalues of $H$ for root vectors with the highest or the lowest roots
are certain value.

First, the modules which contain root spaces in $L$ and $-L$ can be seen from
previous contents.
Among the positive root spaces not in $L$,
those with the smallest eigenvalue in $H$ are considered together.
Let $\lambda_H$ be their eigenvalue and $d_H$ be their dimensions.
Suppose $p_H$ modules which contain space with eigenvalue $\lambda_H$
that also contain the root spaces already obtained.
Then there are $d_H - p_H$ lowest weight modules with the root with
eigenvalue $\lambda_H$ as the lowest root.
The multiplicities of modules can be obtained inductively by replacing
$\lambda_H$ with the next smallest eigenvalue of $H$
and performing the same calculation.
Negative root spaces can be classified by the same calculation.
\section{Unitary principal series representation and complementary series representation}
In this section, we consider a module (say $M$) that is neither highest weight module
nor lowest weight module containing a root vector about the root of type B.
We compute whether the module is
a unitary principal series representation or a complementary series representation.
First, we state the following lemma.
\begin{lem}
  \label{principal}
  If $8 \mu \leq -1$, then $M$ is a unitary principal series representation.
  If $8 \mu > -1$, then $M$ is a complementary series representation.
\end{lem}
\begin{proof}
  From \cite[\S II 1.2]{ht}, $M$ is isomorphic to $U (\nu^+, \nu^-)$.
  $U (\nu^+, \nu^-)$ is a $\lisl_2$-module with $H$ eigenvectors
  $\{ v_n \mid n \in \bbZ \}$ as a basis of linear space, such that
  \begin{align*}
    H v_n &= (\nu^+ - \nu^- + 2j) v_n \hsp (n \in \bbZ),\\
    X v_n &= (\nu^+ + n) v_{n + 1},\\
    Y v_n &= (\nu^- - n) v_{n - 1},\\
    8\mu &= (\nu^+ + \nu^- - 1)^2 - 1.
  \end{align*}
  From \cite[\S III Theorem 1.1.3]{ht},
  if $\nu^+ + \ovl{\nu^-} = 1$, $U (\nu^+, \nu^-)$ is
  a unitary principal series representation.
  When $8 \mu \leq -1$, from
  \begin{align*}
    \lambda &= \nu^+ - \nu^- \in \bbR,\\
    8\mu &= (\nu^+ + \nu^- - 1)^2 - 1 < -1,
  \end{align*}
  using $b \in \bbR$ we can write
  \begin{align*}
    \nu^+ - \nu^- &= \lambda,\\
    \nu^+ + \nu^- &= 1 + bi. \hsp (i = \sqrt{-1})
  \end{align*}
  In this case, we have
  \begin{align*}
    \nu^+ + \ovl{\nu^-} &= \frac{\lambda + 1}{2} + \frac{b}{2}i
    + \frac{- \lambda + 1}{2} - \frac{b}{2}i\\
    &= 1.
  \end{align*}
  Therefore, $M$ is a unitary principal series representation.

  Consider the case when $8 \mu > -1$.
  From \cite[\S III Theorem 1.1.3]{ht},
  if $\nu^\pm \in \bbR$ and if
  $\nu^- - 1$ and $- \nu^+$ are both contained in the interval
  $(l - 1, l)$ with some $l \in \bbZ$, then $U (\nu^+, \nu^-)$ is a complementary series representation.
  From $8 \mu > -1$, we have
  \begin{align*}
    \nu^+ + \nu^- &= 1 \pm \sqrt{8 \mu + 1},\\
    \nu^+ - \nu^- &= \lambda.
  \end{align*}
  Therefore, we have
  \begin{align*}
    - \nu^+, \nu^- - 1 = \frac{- \lambda - 1 \pm \sqrt{8 \mu + 1}}{2}.
  \end{align*}
  We show that they are in $(-1, 0)$.

  We show first that $0 < \lambda < 1$.
  Let $n, m$ be integers such that $n > m \geq 0$.
  We can write
  \begin{align*}
    \lambda = \frac{2 (F_{m + 1} + F_m)}{F_{n + 1} + F_n}.
  \end{align*}
  It is clear that $\lambda > 0$.
  From $a \geq 3$, for integer $z \geq 0$, we have
  \begin{align*}
    F_{z + 2} &= a F_{z + 1} - F_z\\
    &> (a - 1) F_{z + 1}\\
    &\geq 2 F_{z + 1}.
  \end{align*}
  Hence we have
  \begin{align*}
    \frac{F_{m + 1} + F_m}{F_{n + 1} + F_n} < \frac{1}{2}.
  \end{align*}
  Therefore, we have $\lambda < 1$.
  We show that
  \begin{align*}
    -1 < \frac{- \lambda - 1 + \sqrt{8 \mu + 1}}{2}.
  \end{align*}
  From $\lambda < 1$, we have $-1 < \frac{- \lambda - 1}{2}$.
  Therefore, this inequality is shown.
  Next we show
  \begin{align*}
    \frac{- \lambda - 1 + \sqrt{8 \mu + 1}}{2} < 0.
  \end{align*}
  We have $8 \mu = \lambda (\lambda - 2) + 4 k_0$.
  From $0 < \lambda < 1$, we have $\lambda (\lambda - 2) < 0$.
  Also, since $k_0 < 0$, we have $8 \mu < 0$. Therefore,
  we have $\sqrt{8 \mu + 1} < 1$.
  Using $0 < \lambda$ again, we know that
  \begin{align*}
    \frac{- \lambda - 1 + \sqrt{8 \mu + 1}}{2} < 0.
  \end{align*}
  For
  \begin{align*}
    \frac{- \lambda - 1 - \sqrt{8 \mu + 1}}{2} < 0,
  \end{align*}
  this is clear from $\lambda > 0$.
  Finally, we show
  \begin{align*}
    -1 < \frac{- \lambda - 1 - \sqrt{8 \mu + 1}}{2}.
  \end{align*}
  From $k_0 < 0$ and $8 \mu = \lambda^2 - 2 \lambda + 4 k_0$,
  we have $\lambda^2 - 2 \lambda > 8 \mu$.
  From this and $\lambda < 1$ we get $1 - \lambda > \sqrt{8 \mu + 1}$,
  which can be transformed to
  \begin{align*}
    -1 < \frac{- \lambda - 1 - \sqrt{8 \mu + 1}}{2}.
  \end{align*}
  From the above, $\frac{- \lambda - 1 \pm \sqrt{8 \mu + 1}}{2}$
  are both in $(-1, 0)$. Therefore, $M$ is
  a complementary series representation.
\end{proof}
Hereafter, we want to determine when $M$ is complementary series.
First, we consider the case where $i = j$.
we have
\begin{align}
  \tag{B1}\label{8mu}
  \begin{aligned}
    8 \mu &= \lambda^2 - 2 \lambda + 4 k_0,\\
    \lambda &= \frac{2 (F_{n + 1} + F_n)}{F_{i + 1} + F_i},\\
    k_0 &= \frac{-2 (2 F_{i + 1 - n} - a F_{i - n})}{a (F_i^2 + F_{i + 1}^2) - 4 F_i F_{i + 1} - 2},
  \end{aligned}
\end{align}
where $n$ is an integer such that $i > n \geq 0$.
That is, $8 \mu$ is determined by $i, n$, and $a$.
We show that $8 \mu$ is greater than $-1$ with finite exceptions.
\begin{lem}
  \label{MonotonicalForN_IEqualsJ}
  We assume $i = j$. If we consider $8 \mu$ to be a function of $n$
  by \cref{8mu}, $8 \mu$ is monotonically decreasing
  with respect to $n$.
\end{lem}
\begin{proof}
  $\lambda$ is monotonically increasing with respect to $n$.
  Since $8 \mu = \lambda (\lambda - 2) + 4 k_0$
  and $0 < \lambda < 1$, we know that
  $\lambda (\lambda - 2)$ is monotonically decreasing with respect to $n$.
  Since $\{ F_k \}$ is monotonically increasing with respect to $k$,
  $4 k_0$ is monotonically decreasing with respect to $n$.
  Therefore, $8 \mu$ is monotonically decreasing with respect to $n$.
\end{proof}
To show that $8 \mu$ is greater than $-1$ with finite exceptions,
we need to examine when $n$ is large.
\begin{lem}
  \label{MonotonicalForI_IEqualsJ}
  We assume $i = j ,\; n = i - 1$. If we consider $8 \mu$ to be a function of $i$
  by \cref{8mu}, $8 \mu$ is monotonically increasing
  with respect to $i$.
\end{lem}
\begin{proof}
  First we write $\{ F_i \}$ explicitly as follows.
  The real solutions of $x^2 - ax + 1 = 0$ are $x = \frac{a \pm \sqrt{a^2 - 4}}{2}$.
  As $\alpha = \frac{a - \sqrt{a^2 - 4}}{2}$, $\beta = \frac{a + \sqrt{a^2 - 4}}{2}$,
  we can write
  \begin{align*}
    F_i = \frac{\beta^i - \alpha^i}{\beta - \alpha}.
  \end{align*}
  From $n = i - 1$, we have
  \begin{align*}
    \lambda &= \frac{2 (F_{i} + F_{i - 1})}{F_{i + 1} + F_i},\\
    k_0 &= \frac{-2a}{a (F_i^2 + F_{i + 1}^2) - 4 F_i F_{i + 1} - 2}.
  \end{align*}
  Let $t$ be a real variable. We define the functions $\Lambda_1$ and $K_{01}$ as follows.
  \begin{align*}
    \Lambda_1(t) &= \frac{2 (\beta^t - \alpha^t + \beta^{t - 1} - \alpha^{t - 1})}{\beta^{t + 1} - \alpha^{t + 1} + \beta^t - \alpha^t},\\
    K_{01}(t) &= \frac{-2a(\beta - \alpha)^2}{a ((\beta^t - \alpha^t)^2 + (\beta^{t + 1} - \alpha^{t + 1})^2) - 4 (\beta^t - \alpha^t) (\beta^{t + 1} - \alpha^{t + 1}) - 2(\beta - \alpha)^2}\\
    &= \frac{-2a}{\beta^{2t + 1} + \alpha^{2t + 1} - 2}.
  \end{align*}
  We have $\lambda = \Lambda_1(i)$ and $k_0 = K_{01}(i)$.
  Using these function, we can calculate as follows.
  \begin{align*}
    \frac{d}{dt} \Lambda_1
    &= \frac{4 \log \beta (a + 2) (\beta - \alpha)}{(\beta^{t + 1} - \alpha^{t + 1} + \beta^t - \alpha^t)^2}\\
    \frac{d}{dt} (\Lambda_1^2 - 2\Lambda_1)
    &= \frac{8 \log \beta (a + 2) (\beta - \alpha) \left( (1 - a) \beta^t - (1 - a) \alpha^t + 3\beta^{t - 1} - 3\alpha^{t - 1} \right)}{(\beta^{t + 1} - \alpha^{t + 1} + \beta^t - \alpha^t)^3}\\
    \frac{d}{dt} K_{01}
    &= \frac{4a \log \beta (\beta^{2t + 1} - \alpha^{2t + 1})}{(\beta^{2t + 1} + \alpha^{2t + 1} - 2)^2}
  \end{align*}
  Clearing the denominator, we can calculate as follows.
  \begin{align*}
    &\frac{(\beta^{t + 1} - \alpha^{t + 1} + \beta^t - \alpha^t)^3 (\beta^{2t + 1} + \alpha^{2t + 1} - 2)^2}{8 (a + 2) \log \beta} \cdot \frac{d}{dt} (\Lambda_1^2 - 2\Lambda_1 + 4 K_{01})\\
    &= (a + 1) (\beta^{5t + 3} + \alpha^{5t + 3}) + (2a + 3) (\beta^{5t + 2} + \alpha^{5t + 2})\\
    &+ (a - 1) (\beta^{5t + 1} + \alpha^{5t + 1}) - 3 (\beta^{5t} + \alpha^{5t})\\
    &- 3 (\beta^{3t + 4} + \alpha^{3t + 4}) + (a - 1) (\beta^{3t + 3} + \alpha^{3t + 3})\\
    &- (2a + 1) (\beta^{3t + 2} + \alpha^{3t + 2}) - (7a + 11) (\beta^{3t + 1} + \alpha^{3t + 1})\\
    &+ (-4a + 4) (\beta^{3t} + \alpha^{3t}) + 12 (\beta^{3t - 1} + \alpha^{3t - 1})\\
    &+ 12 (\beta^{t + 3} + \alpha^{t + 3}) + (-4a + 4) (\beta^{t + 2} + \alpha^{t + 2})\\
    &- (2a + 6) (\beta^{t + 1} + \alpha^{t + 1}) + (8a + 32) (\beta^t + \alpha^t)\\
    &- (12a + 6) (\beta^{t - 1} + \alpha^{t - 1}).
  \end{align*}
  The coefficient on the left hand side is positive. Using the fact that $\beta^t + \alpha^t$ is monotonically increasing
  and $a \geq 3$,
  we can calculate that the right hand side is also positive.
  This shows that $8 \mu = (\Lambda_1^2 - 2 \Lambda_1 + 4 K_{01})(i)$ is monotonically increasing
  with respect to $i$.
\end{proof}
From \cref{MonotonicalForI_IEqualsJ}, we consider the case when $i = 1 ,\; n = 0$.
\begin{lem}
  \label{MonotonicalForA_IEqualsJ}
  We assume $i = j = 1$ and $n = 0$. If we consider $8 \mu$ to be a function of $a$
  by \cref{8mu}, $8 \mu$ is monotonically increasing
  with respect to $a$.
\end{lem}
\begin{proof}
  Under this assumption, we have
  \begin{align*}
    8 \mu = \frac{-4a^2}{a^3 - 3a - 2}.
  \end{align*}
  Differentiating this as a function of the real variable $a$,
  from $a \geq 3$, we know that
  $8 \mu$ is monotonically increasing
  with respect to $a$.
\end{proof}
\begin{lem}
  \label{AlmostAllsl2ModulesAreComplementary_IEqualsJ}
  When $i = j$, we consider $\lis$-modules of $\lig$ that are neither
  a highest weight module nor a lowest weight module containing a root vector about the root of type B
  obtained by \cref{classification}.
  The modules are complementary series representations,
  except for the following three types. For these exceptions,
  the modules are unitary principal series representations.
  \begin{align*}
    (a, i, n) = (4, 1, 0), (3, 1, 0), (3, 2, 1)
  \end{align*}
\end{lem}
\begin{proof}
  We use \cref{MonotonicalForN_IIsJMinus1}, \cref{MonotonicalForI_IEqualsJ}, and \cref{MonotonicalForA_IEqualsJ}.

  First, when $a = 5 ,\; i = 1 ,\; n = 0$, we have
  $8 \mu > -1$. Therefore, when $a \geq 5$,
  for any $i, n$, the module for $a, i, n$ is a complementary
  series representation.

  Next, when $a = 4, i = 1, n = 0$, we have $8 \mu = - 1.28 < -1$.
  Hence the module for this is a unitary principal series representation.
  On the other hand, when $a = 4, i = 2, n = 1$,
  we have $8 \mu > -1$. Therefore, when $a = 4$,
  the module for $a, i, n$ is a complementary series representation
  except when $i = 1, n = 0$.

  Finally, when $a = 3$, $8 \mu < -1$ when $i = 1, 2$ and $n = i - 1$,
  and in these four cases the module is a unitary principal series representation.
  When $n = i - 2$ or $i = 3$, we have $8 \mu > -1$.
  Therefore, we know that the module is a complementary series representation
  in other cases.

  From the above, with three exceptions,
  neither a highest weight module nor a lowest weight module containing a root vector about the root of type B
  is a complementary series representation.
\end{proof}
Next, we consider the case $i = j - 1$.
In this case, we have
\begin{align}
  \tag{B2}\label{8mu_IIsJMinus1}
  \begin{aligned}
    8 \mu &= \lambda^2 - 2 \lambda + 4 k_0,\\
    \lambda &= \frac{2 F_{n + 1}}{F_{i + 1}},\\
    k_0 &= \frac{-2 (2 F_{i + 1 - n} - a F_{i - n})}{a (F_i^2 + F_{i + 1}^2) - 4 F_i F_{i + 1} - 2},
  \end{aligned}
\end{align}
with $n$ as an integer such that $0 \leq n < i$.
As with $i = j$,
$8 \mu$ is determined by $i, n$ and $a$.
Similar to \cref{MonotonicalForN_IEqualsJ}, we have this lemma.
\begin{lem}
  \label{MonotonicalForN_IIsJMinus1}
  We assume $i = j - 1$. If we consider $8 \mu$ to be a function of $n$
  by \cref{8mu_IIsJMinus1}, $8 \mu$ is monotonically decreasing
  with respect to $n$. \qed
\end{lem}
We consider whether $8 \mu$ is monotonically increasing
with respect to $i$ when $n = i - 1$.
\begin{lem}
  \label{MonotonicalForI_IIsJMinus1}
  We assume $i = j - 1 ,\; n = i - 1$. If we consider $8 \mu$ to be a function of $i$
  by \cref{8mu_IIsJMinus1}, $8 \mu$ is monotonically increasing
  with respect to $i$.
\end{lem}
\begin{proof}
  In this case, we have
  \begin{align*}
    \lambda &= \frac{2 F_{i}}{F_{i + 1}},\\
    k_0 &= \frac{-2a}{a (F_i^2 + F_{i + 1}^2) - 4 F_i F_{i + 1} - 2}.
  \end{align*}
  Let $t$ be a real variable. We define the functions $\Lambda_2$ and $K_{02}$ as follows.
  \begin{align*}
    \Lambda_2(t) &= \frac{2 (\beta^t - \alpha^t)}{\beta^{t + 1} - \alpha^{t + 1}},\\
    K_{02}(t) &= \frac{-2a(\beta - \alpha)^2}{a ((\beta^t - \alpha^t)^2 + (\beta^{t + 1} - \alpha^{t + 1})^2) - 4 (\beta^t - \alpha^t) (\beta^{t + 1} - \alpha^{t + 1}) - 2(\beta - \alpha)^2}\\
    &= \frac{-2a}{\beta^{2t + 1} + \alpha^{2t + 1} - 2}.
  \end{align*}
  We have $\lambda = \Lambda_2(i)$ and $k_0 = K_{02}(i)$.
  Using these function, we can calculate as follows.
  \begin{align*}
    \frac{d}{dt} \Lambda_2
    &= \frac{4 \log \beta (\beta - \alpha)}{(\beta^{t + 1} - \alpha^{t + 1})^2}\\
    \frac{d}{dt} (\Lambda_2^2 - 2\Lambda_2)
    &= \frac{8 \log \beta (\beta - \alpha) \left( 2 \beta^t - 2 \alpha^t - \beta^{t + 1} + \alpha^{t + 1} \right)}{(\beta^{t + 1} - \alpha^{t + 1})^3}\\
    \frac{d}{dt} K_{02}
    &= \frac{4a \log \beta (\beta^{2t + 1} - \alpha^{2t + 1})}{(\beta^{2t + 1} + \alpha^{2t + 1} - 2)^2}
  \end{align*}
  Clearing the denominator, we can calculate as follows.
  \begin{align*}
    &\frac{(\beta^{t + 1} - \alpha^{t + 1})^3 (\beta^{2t + 1} + \alpha^{2t + 1} - 2)^2}{8 \log \beta} \cdot \frac{d}{dt} (\Lambda_2^2 - 2\Lambda_2 + 4 K_{02})\\
    &= (2a - 1) (\beta^{5t + 4} + \alpha^{5t + 4}) + 2 (\beta^{5t + 3} + \alpha^{5t + 3})\\
    &+ (\beta^{5t + 2} + \alpha^{5t + 2}) -2 (\beta^{5t + 1} + \alpha^{5t + 1})\\
    &+ 2 (\beta^{3t + 3} + \alpha^{3t + 3}) - (6a + 7) (\beta^{3t + 2} + \alpha^{3t + 2})\\
    &- 2 (\beta^{3t + 1} + \alpha^{3t + 1}) + 7 (\beta^{3t} + \alpha^{3t})\\
    &- (2a - 2) (\beta^{t + 2} + \alpha^{t + 2}) + 8 (\beta^{t + 1} + \alpha^{t + 1})\\
    &+ (6a - 2) (\beta^t + \alpha^t) - 8 (\beta^{t - 1} + \alpha^{t - 1}).
  \end{align*}
  The coefficient on the left hand side is positive. Using the fact that $\beta^t + \alpha^t$ is monotonically increasing
  and $a \geq 3$,
  we can calculate that the right hand side is also positive.
  This shows that $8 \mu = (\Lambda_2^2 - 2 \Lambda_2 + 4 K_{02})(i)$ is monotonically increasing
  with respect to $i$.
\end{proof}
\begin{lem}
  \label{MonotonicalForA_IIsJMinus1}
  We assume $i = 1, j = 2$, and $n = 0$. If we consider $8 \mu$ to be a function of $a$
  by \cref{8mu_IIsJMinus1}, $8 \mu$ is monotonically increasing
  with respect to $a$.
\end{lem}
\begin{proof}
  Under this assumption, we have
  \begin{align*}
    8 \mu = \frac{-4 (a^4 + a^3 - 3a^2 + a + 2)}{a^5 - 3a^3 - 2a^2}.
  \end{align*}
  Differentiating this as a function of the real variable $a$,
  from $a \geq 3$, we know that
  $8 \mu$ is monotonically increasing
  with respect to $a$.
\end{proof}
\begin{lem}
  \label{AlmostAllsl2ModulesAreComplementary_IIsJMinus1}
  When $i = j - 1$, We consider $\lis$-modules containing a root vector about the root of type B
  that are neither highest weight modules nor lowest weight modules
  obtained by \cref{classification}. The modules are complementary
  series representations, except for the following 4 types.
  For these exceptions,
  the modules are unitary principal series representations.
  \begin{align*}
    (a, i, n) = (5, 1, 0), (4, 1, 0), (3, 1, 0), (3, 2, 1)
  \end{align*}
\end{lem}
\begin{proof}
  We use \cref{MonotonicalForN_IIsJMinus1}, \cref{MonotonicalForI_IIsJMinus1}, and
  \cref{MonotonicalForA_IIsJMinus1}.
  First, when $a = 6, i = 0, n = 0$, $8 \mu > -1$.
  Therefore, when $a \geq 6$, the modules for $a, i, n$ are
  complementary series representations.

  When $a = 5$, if $(a, i, n) = (5, 0, 0)$, then $8 \mu \leq -1$ and the modules
  are unitary principal series representations, and the others are
  complementary series representations.

  When $a = 4$, if $(a, i, n) = (4, 0, 0)$, then the modules
  are unitary principal series representations, and the others are
  complementary series representations.

  When $a = 3$, if $(a, i, n) = (3, 1, 0), (3, 2, 1)$,
  then the modules are unitary principal series representations, and the others are
  complementary series representations.

  From the above, 4 unitary principal series representations are obtained,
  and the rest are all complementary series representations.
\end{proof}
Finally, we consider when $i = j + 1$.
In this case, we have
\begin{align}
  \tag{B3}\label{8mu_IIsJPlus1}
  \begin{aligned}
    8 \mu &= \lambda^2 - 2 \lambda + 4 k_0,\\
    \lambda &= \frac{2 F_n}{F_i},\\
    k_0 &= \frac{-2 (2 F_{i + 1 - n} - a F_{i - n})}{a (F_i^2 + F_{i + 1}^2) - 4 F_i F_{i + 1} - 2},
  \end{aligned}
\end{align}
with $n$ as an integer such that $0 \leq n < i$.
As with $i = j$,
$8 \mu$ is determined by $i, n$ and $a$.
Similar to \cref{MonotonicalForN_IEqualsJ}, we have this lemma.
\begin{lem}
  \label{MonotonicalForN_IIsJPlus1}
  We assume $i = j + 1$. If we consider $8 \mu$ to be a function of $n$
  by \cref{8mu_IIsJPlus1}, $8 \mu$ is monotonically decreasing
  with respect to $n$. \qed
\end{lem}
In the following, we consider the case when $n = i + 1$.
In this case, $8 \mu$ is not monotonically increasing
with respect to $i$.
Let $t$ be a real variable. We define the functions $\Lambda_3$ and $K_{03}$ as follows.
\begin{align*}
  \Lambda_3(t) &= \frac{2 (\beta^{t - 1} - \alpha^{t - 1})}{\beta^t - \alpha^t},\\
  K_{03}(t) &= \frac{-2a(\beta - \alpha)^2}{a ((\beta^t - \alpha^t)^2 + (\beta^{t + 1} - \alpha^{t + 1})^2) - 4 (\beta^t - \alpha^t) (\beta^{t + 1} - \alpha^{t + 1}) - 2(\beta - \alpha)^2}\\
  &= \frac{-2a}{\beta^{2t + 1} + \alpha^{2t + 1} - 2}.
\end{align*}
We have $\lambda = \Lambda_3(i)$ and $k_0 = K_{03}(i)$.
$\Lambda_3$ is monotonically increasing with respect to $t$.
We have $\Lambda_3(1) = 0$ and
\begin{align*}
  \lim_{t \to \infty} \Lambda_3(t) = \frac{2}{\beta} = \frac{4}{a + \sqrt{a^2 - 4}}.
\end{align*}
Since $F_{i + 1} > 2 F_i$, we have $0 < \Lambda_3 < 1$.
Therefore, $\Lambda_3^2 - 2 \Lambda_3$ is monotonically decreasing.
$(\Lambda_3^2 - 2 \Lambda)(1) = 0$ and
\begin{align*}
  \lim_{t \to \infty} (\Lambda_3^2 - 2 \Lambda_3)(t) = \frac{4 - 4 \beta}{\beta^2}.
\end{align*}
Considering when $a = 3$, we have
\begin{align*}
  \Lambda_3^2 - 2 \Lambda_3 (t) &> \frac{4}{\beta^2} - \frac{4}{\beta}\\
  &= \frac{-4 - 4 \sqrt{5}}{7 + 3 \sqrt{5}} = -0.9442719 \cdots.
\end{align*}
$K_{03}$ is monotonically increasing with respect to $t$.
$K_{03}(1) = \frac{-8a}{a^3 - 3a - 2}$ and $\lim_{t \to \infty} K_{03}(t) = 0$.
When $t$ becomes large enough and $K_{03}(t)$ becomes greater than -0.01,
We have $(\Lambda_3^2 - 2\Lambda_3 + K_{03})(t) > -1$.
Therefore, when $i$ is large enough, we have $8 \mu > -1$.
There are only a finite number of $(a, i, n)$s
such that $k_0$ is less than -0.01.
Calculating all of these cases, we have the following lemma.
\begin{lem}
  \label{AlmostAllsl2ModulesAreComplementary_IIsJPlus1}
  When $i = j + 1$, We consider $\lis$-modules containing a root vector about the root of type B
  that are neither highest weight modules nor lowest weight modules
  obtained by \cref{classification}. The modules are complementary
  series representations, except for the following 2 types.
  For these exceptions,
  the modules are unitary principal series representations.
  \begin{align*}
    (a, i, n) = (3, 1, 0), (3, 2, 1)
  \end{align*}
  \qed
\end{lem}
\begin{thm}
  \label{AlmostAllsl2ModulesAreComplementary}
  We consider modules obtained by \cref{classification1} of \cref{classification}.
  The modules are neither highest weight modules nor lowest weight modules
  and contain root vectors about roots of type B.
  The modules are complementary series representations,
  except those enumerated by \cref{AlmostAllsl2ModulesAreComplementary_IEqualsJ},
  \cref{AlmostAllsl2ModulesAreComplementary_IIsJMinus1}
  and \cref{AlmostAllsl2ModulesAreComplementary_IIsJPlus1}.
  For the exceptions, the modules are unitary principal series representations.
\end{thm}
\begin{proof}
  It can be shown from \cref{AlmostAllsl2ModulesAreComplementary_IEqualsJ},
  \cref{AlmostAllsl2ModulesAreComplementary_IIsJMinus1}
  and \cref{AlmostAllsl2ModulesAreComplementary_IIsJPlus1}.
\end{proof}
\section*{Acknowledgements}
I would like to express my appreciation to my supervisor,
Prof. Hisayosi Matumoto for his thoughtful guidance.

\end{document}